\documentclass[11pt, a4paper]{amsart}

\usepackage{setspace, enumerate, amsthm, amsmath, amssymb, hyperref, cite}

\DeclareMathOperator*{\dist}{dist}
\DeclareMathOperator*{\supp}{supp}
\DeclareMathOperator*{\R}{Re}
\DeclareMathOperator*{\I}{Im}

\newcommand{\ep}{\varepsilon}
\newcommand{\s}{\sigma}
\newcommand{\dd}{\mathrm{d}}
\newcommand{\Mlog}{M_\mathrm{log}}
\newcommand{\Mmax}{M_\mathrm{max}}
\newcommand{\B}{\mathcal{B}}
\newcommand{\F}{\mathcal{F}}
\newcommand{\RR}{\mathbb{R}}
\newcommand{\CC}{\mathbb{C}}
\newcommand{\NN}{\mathbb{N}}
\newcommand{\ZZ}{\mathbb{Z}}
\newcommand{\T}{(T(t))_{t\ge0}}
\newcommand{\inv}{^{-1}}

\newtheorem{thm}{Theorem}[section]
\newtheorem{prp}[thm]{Proposition}
\newtheorem{lem}[thm]{Lemma}

\theoremstyle{definition}
\newtheorem{rem}[thm]{Remark}
\newtheorem{ex}[thm]{Example}

\numberwithin{equation}{section}

\begin{document}

\title[Optimal  decay for operator semigroups on Hilbert spaces]{Optimal rates of decay for operator semigroups on Hilbert spaces}

\author{Jan Rozendaal}
\address[J. Rozendaal]{Mathematical Sciences Institute\\ Australian National University\\ Canberra ACT 2601\\ Australia\\ and Institute of Mathematics, Polish Academy of Sciences\\
ul.~\'{S}niadeckich 8\\
00-656 Warsaw\\
Poland}
\email{janrozendaalmath@gmail.com}

\author{David Seifert}
\address[D. Seifert]{St John's College\\ St Giles\\ Oxford OX1 3JP\\ United Kingdom}
\email{david.seifert@sjc.ox.ac.uk}

\author{Reinhard Stahn}
\address[R. Stahn]{Formerly at the Institut f\"ur Analysis, Technische Universit\"at Dresden\\
01062 Dresden, Germany}
\email{reinhardstahn@t-online.de}

\begin{abstract}

We investigate rates of decay for $C_0$-semigroups on Hilbert spaces under assumptions on the resolvent growth of the semigroup generator. Our main results show that one obtains the best possible estimate on the rate of decay, that is to say an upper bound which is also known to be a lower bound, under a comparatively mild assumption on the growth behaviour. This extends several statements obtained by Batty, Chill and Tomilov (J.\ Eur.\ Math.\ Soc., vol.\ 18(4), pp.\ 853--929, 2016). In fact, for a large class of semigroups our condition is not only sufficient but also necessary for this optimal estimate to hold. Even without this assumption we obtain a new quantified asymptotic result which in many cases of interest gives a sharper estimate for the rate of decay than was previously available, and for semigroups of normal operators we are able to describe the asymptotic behaviour exactly. We illustrate the strength of our theoretical results by using them to obtain sharp estimates on the rate of energy decay for a wave equation subject to viscoelastic damping at the boundary. 

\end{abstract}


\maketitle

\section{Introduction}\label{sec:intro} 

Motivated by applications to partial differential equations, and in particular to the study of energy decay in damped wave equations, there has been a considerable amount of interest over the last decade in obtaining sharp estimates for the asymptotic behaviour of $C_0$-semigroups. Given a complex Banach space $X$, consider the abstract Cauchy problem 
\begin{equation}\label{eq:ACP_int}
\left\{
\begin{aligned}
\dot{z}(t)&=Az(t),\quad  t\ge0,\\
z(0)&=x,
\end{aligned}
\right.
\end{equation}
where $A$ is a closed and densely defined operator on $X$ and $x\in X$ is the initial data. Let us suppose that \eqref{eq:ACP_int} is well-posed in the sense that $A$ is the infinitesimal generator of a $C_0$-semigroup $\T$ on $X$, and let us assume that the semigroup $\T$ is \emph{bounded}, which is to say that $\sup_{t\ge0}\|T(t)\|<\infty$. Then the unique solution $z\colon\RR_+\to X$ of \eqref{eq:ACP_int} in the \emph{mild} sense is given by $z(t)=T(t)x$, $t\ge0$, and $z$ solves \eqref{eq:ACP_int} in the \emph{classical} sense if and only if $x$ lies in the domain of $A$. In applications the norm of $X$ often has a useful physical interpretation, for instance as an energy. Since $\T$ is assumed to be bounded the spectrum of $A$ necessarily lies in the closed left-half plane, and in many applications it is even contained in the open left-half plane. In this case $A$ is invertible and its domain coincides with the range of $A\inv$, so in order to obtain (uniform) rates of decay for classical solutions one is led to investigate the quantitative behaviour of the operator norm $\|T(t)A\inv\|$ as $t\to\infty$. 

Ever since the pioneering work of Lebeau \cite{Leb96} one of the central objectives in the asymptotic theory of $C_0$-semigroups has been to obtain good estimates for the rate at which this quantity decays assuming one has knowledge of how the resolvent operator $R(is,A)=(isI-A)\inv$, $s\in\RR$, behaves along the imaginary axis. The underlying motivation here is that in typical applications estimates for the norm of the resolvent are more or less readily available whereas information on the semigroup itself is hard to come by. Let $M(s)=\smash{\sup_{|r|\le s}}\|R(is,A)\|$, $s\ge0$, and suppose that $M(s)\to\infty$ as $s\to\infty$. It was shown in \cite{Batty-Duyckaerts08} that 
\begin{equation}\label{eq:BD_int}
\frac{c}{M\inv(Ct)}\le\|T(t)A\inv\|\le \frac{C}{\Mlog\inv(ct)}
\end{equation}
for some constants $C,c>0$ and  all sufficiently large values of $t$, where $M\inv$ is any right-inverse of $M$ and $\Mlog$ is a modified version of the function $M$ which grows faster than $M$ itself by a logarithmic correction factor. For instance, if $M$ grows like $s^\alpha$ as $s\to\infty$ for some $\alpha>0$, a case which had previously been considered in \cite{BEPS06}, then \eqref{eq:BD_int} becomes
\begin{equation}\label{eq:gap}
\frac{c}{t^{1/\alpha}}\le\|T(t)A\inv\|\le C\left(\frac{\log t}{t}\right)^{1/\alpha},\quad t\ge1,
\end{equation}
and the authors of \cite{Batty-Duyckaerts08} conjectured that in this case ``the logarithmic correction may be dropped, or at least replaced by a smaller rectification, in the case of Hilbert space, but cannot be forgotten in general Banach spaces."  Both parts of this conjecture were proved to be correct in the highly influential paper \cite{BoTo10}, whose authors showed that the upper bound in \eqref{eq:gap} cannot be improved if no restrictions are imposed on the Banach space $X$, whereas if $X$ is assumed to be a  Hilbert space then the logarithm in \eqref{eq:gap} may be dropped completely. The latter result has been applied extensively in the recent literature on energy decay for damped wave equations  and other concrete partial differential equations; see for instance \cite{AbNi15, AnLe14, AvLaTr16, BaPaSe16, BuZu16, CaCaTe17, Gu17, HaLiYo15, HaZu16, LeLe17, LiZh15, LiZh16, OqPa17, Sta17b, Sta17c} and also \cite[Section~1]{BaChTo16}.

If $M$ is no longer assumed to grow polynomially then it is not difficult to see that one cannot always expect the lower bound in \eqref{eq:BD_int} to coincide with the actual rate of decay of $\|T(t)A\inv\|$ as $t\to\infty$, even when $X$ is a Hilbert space; see  for instance \cite[Example~5.2]{BaChTo16} and the discussion following Remark~\ref{rem:DC} below. It is natural to ask, therefore, for which functions $M$ beyond polynomials \emph{is} it possible, at least in the Hilbert space setting, to replace $\smash{\Mlog\inv}$ by $M\inv$  in \eqref{eq:BD_int}. This question was first addressed in \cite{BaChTo16}, where it is shown that for certain so-called \emph{regularly varying} functions, which in a sense are close to growing polynomially, this is indeed possible. The proof of this result given in \cite{BaChTo16} relies on delicate results from functional calculus theory, and in fact the authors of\cite{BaChTo16} do not obtain the improved estimate for \emph{all} regularly varying functions $M$ but only for a certain subclass. They also show that for normal semigroups one obtains the sharper upper bound if and only if $M$, in the terminology of this paper, has \emph{positive increase}, which is a strictly weaker condition than regularly varying growth. 

The purpose of this paper is to extend the main results of \cite{BaChTo16} by showing that even for general bounded semigroups one  may in fact replace $\smash{\Mlog\inv}$ by $M\inv$ in \eqref{eq:BD_int} for \emph{all} functions $M$ which have positive increase. Since for normal semigroups this condition is not only sufficient but also necessary for the sharper estimate to hold, ours is in a sense the best possible result of this kind. We furthermore investigate rates of decay under milder assumptions on the resolvent growth, and in particular we are able to give an exact description of the rate of decay under arbitrary resolvent growth in the case of normal semigroups. We summarise several of our main results as follows.

\begin{thm}\label{thm:intro}
Let $X$ be a complex Hilbert space and let $A$ be the generator of a bounded $C_0$-semigroup $\T$ on $X$. Suppose that $\s(A)\cap i\RR=\emptyset$ and let $M\colon\RR_+\to(0,\infty)$ be defined by $M(s)=\smash{\sup_{|r|\le s}\|R(ir,A)\|}$, $s\ge0$. If $M$ has positive increase, then there exist constants $C,c>0$ such that
\begin{equation}\label{eq:est_int}
\frac{c}{M\inv(t)}\le\|T(t)A\inv\|\le\frac{C}{M\inv(t)}
\end{equation}
for all sufficiently large values of $t$. Moreover, if $\T$ is a semigroup of normal operators, then the upper bound in \eqref{eq:est_int} holds if and only if $M$ has positive increase, and in fact whenever $M$ is unbounded and $\ep\in(0,1)$ we have
\begin{equation}\label{eq:Mmax_intro}
\frac{1-\ep}{\Mmax\inv(t)}\le\|T(t)A\inv\| \le \frac{1}{\Mmax\inv(t)}
\end{equation}
for all sufficiently large values of $t$, where $\Mmax\colon[1,\infty)\to(0,\infty)$ is defined by $\Mmax(s)=\max_{1\le\lambda\le s}M(\lambda\inv s)\log\lambda,$ $s\ge1.$
\end{thm}

Notice that the lower bound in \eqref{eq:est_int} differs from the lower bound in \eqref{eq:BD_int} in that the former contains only one unspecified constant. As we show in Section~\ref{sec:functions} below, the fact that one may choose the second constant to equal 1 here is intimately connected with the properties of functions having positive increase, and indeed characterises this class of functions. Note further that the function $\Mmax$ in general grows  more slowly than $\Mlog$, so  the rate in \eqref{eq:Mmax_intro} tends to be sharper than the upper bound in \eqref{eq:BD_int}. Of course, if $M(s)$ grows like $s^\alpha$ as $s\to\infty$ then \eqref{eq:Mmax_intro} leads to \eqref{eq:gap} without the logarithm as in \cite{BoTo10} (but now with asymptotic equivalence rather than unknown constants $C,c>0$). On the other hand, if $M$ is a slowly growing function such as the logarithm then $\Mmax$ and $\Mlog$ have essentially the same growth, so the the upper bound in \eqref{eq:BD_int} already gives the correct rate of decay and \eqref{eq:Mmax_intro} is merely a more precise version of this estimate.

The general approach we adopt in obtaining these results is inspired by the proof of \cite[Theorem~4.7]{BaChTo16} but is nevertheless different in spirit from the approach taken in \cite{BaChTo16}. In particular, we do not rely on any  intricate results from functional calculus theory. Instead we combine the basic idea found in the proof of \cite[Theorem~4.7]{BaChTo16} with techniques recently developed in \cite{ChiSei16}. Another important influence on the ideas underlying our approach, although perhaps a less conspicuous one given our focus on the Hilbert space setting, comes from the theory of operator-valued $(L^p,L^q)$ Fourier multipliers and its use in the asymptotic theory of $C_0$-semigroups, as developed in \cite{Rozendaal-Veraar17a, Rozendaal-Veraar17b, RoVe17}. We hope in future work to explore this aspect more fully, also for non-Hilbertian Banach spaces.

Our paper is structured as follows. First, in Section~\ref{sec:functions}, we briefly introduce the requisite background material on regularly varying functions and functions having positive increase. Section~\ref{sec:opt_dec} is the heart of this paper. Here we prove  one of our main results, Theorem~\ref{thm:inf}, which contains the first part of Theorem~\ref{thm:intro} above, namely that for bounded $C_0$-semigroups on Hilbert spaces the rate of decay of $\|T(t)A\inv\|$ as $t\to\infty$ can be estimated from above in terms of $M\inv$ whenever $M$ has positive increase. Following \cite{BaChTo16, ChiSei16, Ma11, Sei15} we also consider the cases where the resolvent operator is allowed to have a singularity not just at infinity but instead at zero, or indeed both at zero and at infinity. In the latter case our result, Theorem~\ref{thm:zero_inf}, is the first in the literature yielding the $M\inv$-estimate for non-polynomially growing resolvents. In each of the three cases we moreover show, as indicated in Theorem~\ref{thm:intro}, that the assumption of positive increase is not only sufficient but also necessary for this sharper estimate to hold, at least in many naturally arising cases and in particular for  semigroups of normal operators. In Section~\ref{sec:quasi} we relax the condition of positive increase. First, in Theorem~\ref{thm:quasi}, we obtain a new quantified asymptotic result for general bounded $C_0$-semigroups on Hilbert spaces, which in many cases improves on the known decay estimates. Then, in Theorem~\ref{thm:normal_rate}, we prove the last part of Theorem~\ref{thm:intro} above, by determining the precise rate of decay for normal semigroups. Finally, in Section~\ref{sec:wave} we consider a one-dimensional wave equation with viscoelastic damping at the boundary and, in particular, we provide a simple criterion for determining whether the rate of energy decay can be estimated from above and below by the same function, namely the reciprocal of  the so-called \emph{acoustic impedance} of the system. We also show, by means of an explicit construction, that this model is rich enough to generate many examples which are covered by our results but not by those found in the previous literature.

Throughout we let $\NN=\{1,2,3,\dotsc\}$, $\ZZ_+=\NN\cup\{0\}$ and $\RR_+=[0,\infty)$. We write $\CC_-=\{z\in\CC:\R z<0\}$ for the open left-half plane. Given functions $f,g\colon[a,\infty)\to(0,\infty)$ for some $a\ge0$ we write $f(t)=O(g(t))$, $t\to\infty$, if there exists a constant $C>0$ such that $f(t)\le Cg(t)$ for all sufficiently large $t\ge a$, and we write $f(t)\asymp g(t)$, $t\to\infty$, if both $f(t)=O(g(t))$ and $g(t)=O(f(t))$ as $t\to\infty$. The functions $f$ and $g$ are said to be \emph{asymptotically equivalent} if $f(t)/g(t)\to1$ as $t\to\infty$, and in this case we write $f(t)\sim g(t)$, $t\to\infty$. Given non-negative quantities $x$ and $y$ we occasionally write $x\lesssim y$ if $x\le Cy$ for  some constant $C>0$. Given a Banach space $X$ we write $\B(X)$ for the algebra of bounded linear operators on $X$. Throughout the remainder of this paper, all Banach spaces are implicitly assumed to be complex. If $A$ is a closed linear operator on $X$ we write $\s(A)$ for the spectrum of $A$, $\rho(A)=\CC\setminus\s(A)$ for its resolvent set, and given $z\in\rho(A)$ we let $R(z,A)=(zI-A)\inv$ denote the resolvent operator. We write $\F $ for the Fourier transform given, for a vector-valued function $h\in L^1(\RR,X)$, by
$$\F h(s)=\int_\RR e^{-ist}h(t)\,\dd t,\quad s\in\RR,$$
and we define the Laplace transform of a function $h\in L^\infty(\RR_+,X)$ by 
$$\widehat{h}(z)=\int_{\RR_+}e^{-z t}h(t)\,\dd t,\quad \R z>0.$$

\section{Special classes of functions}\label{sec:functions}

In this section we present relevant background material on two important classes of real-valued functions. Specifically, we shall introduce the notion of \emph{positive increase}, which will be crucial in what follows, but we begin by discussing the perhaps more widely known class of \emph{regularly varying} functions. Regularly varying functions feature prominently in various classical areas of mathematics, most notably in probability theory, and since the publication of \cite{BaChTo16} they have moreover played an increasingly important role in the quantified asymptotic theory of $C_0$-semigroups. They will appear in various places throughout our paper.

Given $a\ge0$ and $\alpha\in\RR$, we say that a measurable function $M\colon[a,\infty)\to(0,\infty)$ is \emph{regularly varying (of index $\alpha$)} if 
\begin{equation}\label{eq:reg}
\lim_{s\to\infty}\frac{M(\lambda s)}{M(s)}=\lambda^\alpha,\quad \lambda \ge1.
\end{equation}
As is shown in \cite[Theorem~1.4.3]{BiGoTe89}, the mere existence of the limit in \eqref{eq:reg} for all values of $\lambda$ in some subset of $[1,\infty)$ which has positive measure already implies that $M$ is regularly varying. A measurable function $M\colon[a,\infty)\to(0,\infty)$ which is regularly varying of index zero is said to be \emph{slowly varying}. If $a\ge0$, $\alpha\in\RR$ and $M\colon[a,\infty)\to(0,\infty)$ is a regularly varying function of index $\alpha$ then there exists a slowly varying function $\ell\colon[a,\infty)\to(0,\infty)$ and a strictly positive $s_0\ge a$ such that $M(s)= s^\alpha\ell(s)$, $s\ge s_0$. Moreover, by Karamata's Theorem \cite[Theorem 1.3.1]{BiGoTe89} slowly varying functions are precisely those functions $\ell\colon[a,\infty)\to(0,\infty)$ which admit a representation of the form
\begin{equation}\label{eq:slow_rep}
 \ell(s) = q(s)\exp\left( \int_a^s \frac{p(\tau)}{\tau}\, \dd\tau \right),\quad s\ge a,
 \end{equation}
where $p\colon[a,\infty)\rightarrow\RR$ is a measurable function such that $s\mapsto p(s)/s$ is locally integrable on $[a,\infty)$ and $p(s)\rightarrow0$ as $s\rightarrow\infty$, and $q\colon[a,\infty)\rightarrow(0,\infty)$ is a measurable function such that $q(s)\rightarrow q_0$ as $s\rightarrow\infty$ for some $q_0>0$.  Using this representation it can be shown that every regularly varying function of strictly positive (respectively, negative) index is asymptotically equivalent to an eventually increasing (respectively, decreasing) regularly varying function of the same index, and one can even ensure that this function is smooth; see \cite[Theorems~1.5.3 and 1.8.2]{BiGoTe89}. Moreover, if one is interested in regularly varying functions only up to asymptotic equivalence then one may always take the function $q$ in the representation \eqref{eq:slow_rep} to be constant. Further information about regularly varying functions may be found in~\cite[Chapters 1 and 2]{BiGoTe89}, but see also \cite[Section~2]{BaChTo16}. 

Given $a\ge0$ and a measurable function $M\colon[a,\infty)\to(0,\infty)$ we say that $M$ has \emph{positive increase} if there exist strictly positive constants $\alpha>0$, $c\in(0,1]$ and $s_0\geq a$ such that 
\begin{equation}\label{eq:pos_inc}
\frac{M(\lambda s)}{M(s)}\ge c\lambda^\alpha,\quad \lambda\ge1, \,s\ge s_0.
\end{equation}
In particular, if $M$ has positive increase then $M(s)\to\infty$ with at least polynomial speed as $s\to\infty$. Functions of positive increase will play a central role in the remainder of this paper. The following result gives a useful characterisation of functions having positive increase  among all non-decreasing  functions. Note that for such monotonic functions $M$ one may choose \emph{any} strictly positive $s_0\ge a$ in \eqref{eq:pos_inc}. Recall also that monotonic functions are automatically measurable.

\begin{lem}\label{lem:pos_inc}
Let $a\ge0$. If $M\colon[a,\infty)\to(0,\infty)$ is non-decreasing, then $M$ has positive increase if and only if there exists $\lambda>1$ such that
\begin{equation}\label{eq:liminf}
\liminf_{s\to\infty}\frac{M(\lambda s)}{M(s)}>1.
\end{equation}
\end{lem}

\begin{proof}
It is clear that if  $M$ has positive increase then  \eqref{eq:liminf} holds for all sufficiently large $\lambda>1$, even without the monotonicity assumption. Suppose therefore that \eqref{eq:liminf} holds. We may find strictly positive constants $\alpha>0$, $\lambda_0>1$ and $s_0\ge a$ such that 
$$\frac{M(\lambda_0 s)}{M(s)}\ge\lambda_0^\alpha,\quad s\ge s_0.$$
Given $\lambda\ge1$ we may uniquely express $\lambda$ in the form $\lambda=\lambda_0^n\mu$, where $n\in\ZZ_+$ and $1\le\mu<\lambda_0$. Then
$$\frac{M(\lambda s)}{M(s)}\ge\frac{M(\lambda_0^n s)}{M(s)}\ge\lambda_0^{n\alpha}\ge c\lambda^\alpha,\quad s\ge s_0,$$
where $c=\lambda_0^{-\alpha}$,  
so $M$ has positive increase. 
\end{proof}

From Lemma~\ref{lem:pos_inc} and our earlier observations about regularly varying functions and eventual monotonicity we see in particular that, given $a\ge0$, any function $M\colon[a,\infty)\to(0,\infty)$ which is regularly varying with strictly positive index has positive increase. On the other hand, slowly varying functions do not have  positive increase. Note also that the class of functions having positive increase is strictly larger than the class of regularly varying functions with positive index. Indeed, a function may have positive increase without being regularly varying simply because it grows faster than any polynomial, as is the case for $M(s)=e^{\alpha s}$, $s\ge0$, for any $\alpha>0$, but in fact the same phenomenon arises for functions of moderate growth such as $M(s)=s^\alpha(2+\sin s)$, $s\ge1$, again for any $\alpha>0$. Importantly for our purposes, there also exist non-decreasing functions of moderate growth which fail to be regularly varying but nevertheless have positive increase, for instance  $M(s)=s^{2+m(s)}$ with $m(s)=\sin(\log(\log s))$, $s\ge2$. 

In what follows, given $a\ge0$ and a continuous non-decreasing function $M\colon[a,\infty)\to(0,\infty)$ such that $M(s)\to\infty$ as $s\to\infty$, we denote by $M\inv\colon[M(a),\infty)\to[a,\infty)$ its right-continuous right-inverse, given by $M\inv(s)=\sup\{r\ge a:M(r) \le s\}$ for $s\ge M(a)$. The definition implies that $M(M\inv(s))=s$, $s\ge M(a)$, and $M\inv(M(s))\ge s$, $s\ge a$. We conclude this section with a useful observation.

\begin{prp}\label{prp:M_inv}
Let $a\ge0$ and suppose that $M\colon[a,\infty)\to(0,\infty)$ is a continuous non-decreasing function such that $M(s)\to\infty$ as $s\to\infty$. If $M$ has positive increase then for every $c>0$ we have 
\begin{equation}\label{eq:asymp}
M\inv(t)\asymp M\inv(ct),\quad t\to\infty.
\end{equation}
Conversely, if \eqref{eq:asymp} holds for some strictly positive $c\ne1$, then $M$ has positive increase and in particular \eqref{eq:asymp} holds for all $c>0$.
\end{prp}

\begin{proof}
If $M$ has positive increase  then there exist strictly positive constants  $\alpha>0$, $c_0\in(0,1]$ and $s_0\ge a$ such that
\begin{equation}\label{eq:pos_inc_lem}
\frac{M(R)}{M(s)} \geq c_0 \left( \frac{R}{s} \right)^{\alpha}\!,\quad R\geq s\geq s_0.
\end{equation}
Let $t\geq M(s_0)$ and $\lambda\ge1$. Setting $R=M\inv(\lambda t)$ and $s=M\inv(t)$ in \eqref{eq:pos_inc_lem} we see that
\begin{align*}
 \frac{M\inv(\lambda t)}{M\inv(t)} \leq c_0^{-1/\alpha} \lambda^{1/\alpha} .
\end{align*}
Now \eqref{eq:asymp} follows easily using the fact that $M\inv$ is non-decreasing.

Conversely, suppose that \eqref{eq:asymp} holds for some strictly positive $c\ne1$. Let us first assume that $c>1$. Then there exist $\lambda>1$ and $t_0\ge M(a)$ such that $M\inv(ct)\le \lambda M\inv(t)$ for all $t\ge t_0$. If $s=M\inv(t)$ for $t\ge t_0$ then $\lambda s\ge M\inv(ct)$ and hence $M(\lambda s)\ge ct$. It follows that
$$\liminf_{s\to\infty}\frac{M(\lambda s)}{M(s)}\ge c>1,$$
so by Lemma~\ref{lem:pos_inc} the function $M$ has positive increase. A similar argument applies if $c\in(0,1)$, and the final  statement follows from the first part.
\end{proof}

\section{Optimal decay for resolvent growth with positive increase}\label{sec:opt_dec}

\subsection{Singularity at infinity}\label{sec:inf}

The following result is proved in \cite{Batty-Duyckaerts08}.

\begin{thm}\label{thm:BD}
Let $X$ be a Banach space and let $A$ be the generator of a bounded $C_0$-semigroup $\T$ on $X$. Suppose  that $\s(A)\cap i\RR=\emptyset$ and that $M\colon\RR_+\to(0,\infty)$ is a continuous non-decreasing function such that $\|R(is,A)\|\le M(|s|)$, $s\in\RR$. Then there exists a constant $c>0$ such that
\begin{equation}\label{eq:Mlog}
\|T(t)A\inv\|=O\big(\Mlog\inv(ct)\inv\big),\quad t\to\infty,
\end{equation}
where $\Mlog\colon\RR_+\to(0,\infty)$ is defined by $\Mlog(s)=M(s)(\log(1+s)+\log(1+M(s)))$, $s\ge0$. \end{thm}

The spectral assumption is natural here, since by \cite[Proposition~1.3]{Batty-Duyckaerts08} we have $\s(A)\cap i\RR=\emptyset$ whenever $\|T(t)A\inv\|\to0$ as $t\to\infty$.
As discussed in Section~\ref{sec:intro}, the same result implies that if in the setting of Theorem~\ref{thm:BD} we let $M(s)=\smash{\sup_{|r|\le s}}\|R(ir,A)\|$, $s\ge0$, and assume that $M(s)\to\infty$ as $s\to\infty$, then there exist constants $C,c>0$ such that
\begin{equation}\label{eq:lb}
\|T(t)A\inv\|\ge\frac{c}{M\inv(Ct)}
\end{equation}
for all sufficiently large values of $t$.  Our first main result, Theorem~\ref{thm:inf} below, shows that if $X$ is a Hilbert space and $M$ has positive increase then we may replace $\smash{\Mlog\inv}$ by $M\inv$ in \eqref{eq:Mlog}, thus obtaining an upper bound of the same form as the lower bound in \eqref{eq:lb}. This extends \cite[Corollary~5.7]{BaChTo16}, where the corresponding result is obtained for regularly varying functions $M$ satisfying $M(s)= s^\alpha/\ell(s)$, $s\ge1$, for some $\alpha>0$ and some non-decreasing slowly varying function $\ell\colon[1,\infty)\to(0,\infty)$ having a certain symmetry property. It is worth noting that for functions $M$ which grow significantly faster than polynomially the asymptotic behaviour of $M\inv$ is the same as that of $\smash{\Mlog\inv}$. Thus Theorem~\ref{thm:BD} is already optimal in these cases, and our result improves Theorem \ref{thm:BD} only if the growth of $M$ is sufficiently close to being polynomial. 

The proof of Theorem~\ref{thm:inf} combines ideas taken from \cite[Theorem~4.7]{BaChTo16} and \cite{ChiSei16}, and is inspired by techniques from operator-valued Fourier multiplier theory; see \cite{Rozendaal-Veraar17a, Rozendaal-Veraar17b, RoVe17}. More specifically, the first step of our proof is to decompose each relevant semigroup orbit into what may be viewed as a high-frequency component and a low-frequency component. This is achieved by means of a splitting technique found also in \cite{BaChTo16,ChiSei16}. We then estimate the high-frequency component using repeated integration by parts, and we apply Plancherel's theorem to bound the low-frequency component after rewriting it in terms of a certain Fourier multiplier operator. The symbol of this Fourier multiplier operator is determined by the resolvent along the imaginary axis, raised to a particular power which comes out of the condition that $M$ has positive increase. Notice  that the functions $g_0$ and $g_1$ introduced in our proof already played a crucial role in the proof of \cite[Theorem~4.7]{BaChTo16}.

\begin{thm}\label{thm:inf}
Let $X$ be a Hilbert space and let $A$ be the generator of a bounded $C_0$-semigroup $\T$ on $X$. Suppose that $\s(A)\cap i\RR=\emptyset$ and that $M\colon\RR_+\to(0,\infty)$ is a continuous non-decreasing function of positive increase such that $\|R(is,A)\|\le M(|s|)$, $s\in\RR$. Then 
\begin{equation}\label{eq:Minv}
\|T(t)A\inv\|=O\left(M\inv(t)\inv\right),\quad t\to\infty.
\end{equation}
\end{thm}

\begin{proof}
Let $\psi\colon\RR\to\CC$ be a Schwartz function such that $\psi(0)=\|\psi\|_{L^\infty}=1$ and $\supp\psi\subseteq[-1,1]$, and let $\phi=\F\inv\psi$. 
For $R>0$ let $\phi_R(t)=R\phi(Rt)$, $t\in\RR$, and $\psi_R=\F\phi_R$, so that $\psi_R(s)=\psi(R\inv s)$, $s\in\RR$. Note also that $\smash{\int_\RR}\phi_R(t)\,\dd t=1$ for all $R>0$. Now temporarily fix $t>0$ and, given $n\in\ZZ_+$, let $g_n\colon\RR\to\RR$ be defined by
\begin{equation}\label{eq:g}
g_n(s)=\begin{cases}
0, & s<0,\\
s^n,& 0\le s\le t,\\
s^n-(s-t)^n, & s>t.
\end{cases}
\end{equation}
In particular, $g_0=\chi_{[0,t]}$ and $g_n(s)=n\int_0^sg_{n-1}(\tau)\,\dd\tau$ for $n\ge1$ and $s\in\RR$. Let $x\in X$ and $n\in\NN$ be fixed for now. We define the map $h_n\colon\RR\to X$ by $h_n(s)=g_n(s)T(s)A\inv x$, $s\in\RR$, where the semigroup is extended by zero to the whole real line.  Then 
\begin{equation}\label{eq:sg_int}
T(t)A\inv x=\frac{n+1}{t^{n+1}}\int_0^t T(t-s)h_n(s)\,\dd s.
\end{equation}
Our strategy is to split this integral by writing $h_n=(\delta-\phi_R)*h_n+\phi_R*h_n$, where $\delta$ denotes the Dirac mass at zero, and to estimate the resulting two integrals separately by making suitable choices of $R>0$ and of $n\in\NN$.

We begin by introducing the auxiliary function $\Phi\colon\RR\to\RR$ defined by
$$\Phi(s)=\begin{cases}
\int_{-\infty}^s\phi(\tau)\,\dd \tau, & s<0,\\
-\int_s^\infty \phi(\tau)\,\dd \tau, & s\ge0,
\end{cases}$$
so that $\Phi'=\phi-\delta$ in the sense of distributions. Using the fact that $\Phi$, being a primitive of a Schwartz function, decays rapidly at infinity and that $\int_\RR\phi_R(s)\,\dd s=1$, a simple calculation using integration by parts yields
\begin{equation}\label{eq:diff}
(\delta-\phi_R)*h_n(s)=-\frac1R\int_0^\infty \Phi(Rs-\tau)h_n'(R\inv\tau)\,\dd\tau, \quad s\in \RR.
\end{equation}
Now the distributional derivative of $h_n$ is given by 
$$h_n'(s)=ng_{n-1}(s)T(s)A\inv x+g_n(s)T(s)x,\quad s\in\RR,$$
and hence
$$\|h_n'(s)\|\le K(ns^{n-1}+s^n)(\|A\inv\|+1)\|x\|,\quad s\ge0,$$
where $K=\sup_{t\ge0}\|T(t)\|$. It follows from \eqref{eq:diff} that
\begin{equation}\label{eq:diff_est}
\|(\delta-\phi_R)*h_n(s)\|\lesssim\frac{\|x\|}{R}\int_0^\infty |\Phi(Rs-\tau)|\left(n\left(\frac{\tau}{R}\right)^{n-1}+\left(\frac{\tau}{R}\right)^n\right)\,\dd\tau
\end{equation}
for all $s\in\RR$, where the implicit constant is independent of  $R$, $n$, $t$ and $x$.
We now inductively define functions $\Phi_k\colon\RR\to\RR$, $k\in\NN$, by setting $\Phi_1=|\Phi|$ and
\begin{equation}\label{eq:Phi}
\Phi_{k+1}(s)=\begin{cases}
\int_{-\infty}^s\Phi_k(\tau)\,\dd \tau, & s<0,\\
-\int_s^\infty \Phi_k(\tau)\,\dd \tau, & s\ge0,
\end{cases}
\end{equation}
for $k\ge1$. Then, for each $k\in\NN,$ $\Phi_k$ vanishes rapidly at infinity and we have $\Phi'_{k+1}=\Phi_k-\langle\Phi_k\rangle\delta$ in the sense of distributions, where $\langle\Phi_k\rangle=\smash{\int_\RR}\Phi_k(s)\,\dd s$. Hence by a simple inductive argument using integration by parts we see that, for $m\in\ZZ_+$ and $s\ge0$, 
$$\int_0^\infty |\Phi(s-\tau)|\tau^m\,\dd\tau=\sum_{k=0}^{m-1}\frac{m!}{(m-k)!}\langle\Phi_{k+1}\rangle s^{m-k}+m!\int_{-\infty}^s\Phi_{m+1}(\tau)\,\dd\tau,$$
and therefore 
$$\int_0^\infty|\Phi(Rs-\tau)|\left(\frac{\tau}{R}\right)^m\,\dd\tau\le\sum_{k=0}^{m}\frac{m!}{(m-k)!}\|\Phi_{k+1}\|_{L^1}R^{-k}s^{m-k}.$$
Applying this with $m=n-1$ and $m=n$ in \eqref{eq:diff_est} we find after a simple calculation that
\begin{equation*}\label{eq:diff_int_est}
\left\|\frac{n+1}{t^{n+1}}\int_0^tT(t-s)(\delta-\phi_R)*h_n(s)\,\dd s\right\|\lesssim\frac{\|x\|}{R}\Big(P_n(Rt)+\frac{n+1}{t}P_{n-1}(Rt)\Big),
\end{equation*}
where the implicit constant is still independent of $R$, $n$, $t$ and $x$ and where, for $m\in\ZZ_+$, 
\begin{equation}\label{eq:Pm}
P_m(s)=\sum_{k=0}^{m}\frac{(m+1)!}{(m+1-k)!}\frac{\|\Phi_{k+1}\|_{L^1}}{s^{k}},\quad s>0.
\end{equation}

We now turn to the remaining term in the splitting. Note first that by H\"older's inequality
\begin{equation}\label{eq:Holder}
\left\|\frac{n+1}{t^{n+1}}\int_0^tT(t-s)\phi_R*h_n(s)\,\dd s\right\|\le K \frac{n+1}{t^{n+1/2}}\|\phi_R*h_n\|_{L^2(\RR,X)}.
\end{equation}
We now estimate the $L^2$-norm of $\phi_R*h_n$. Given $\alpha>0$, define the function $h_{n,\alpha}\in {L^1(\RR)}$ by  $h_{n,\alpha}(s)=e^{-\alpha s}h_n(s)$, $s\in\RR$. Then $h_{n,\alpha}(s)=n!(T_{\alpha}^{*n}*h_{0,\alpha})(s)$, where $T_{\alpha}(s)=e^{-\alpha s}T(s)$, $s\in\RR$, again after extending the semigroup by zero to the whole real line. Hence 
\begin{equation}\label{eq:FT}
(\F h_{n,\alpha})(s)=n!R(is+\alpha,A)^n\widehat{h_0}(is+\alpha),\quad s\in\RR,
\end{equation}
 and by the dominated convergence theorem, given any Schwartz function $\eta\colon \RR\to\CC$, we have
\begin{align*}\int_\RR \phi_R*h_n(s)\eta(s)\,\dd s&=\lim_{\alpha\to0+}\int_0^\infty h_{n,\alpha}(s)\xi_R(s)\,\dd s\\&=\lim_{\alpha\to0+}\int_\RR \phi_R*h_{n,\alpha}(s)\eta(s)\,\dd s\\
&=\lim_{\alpha\to0+}\int_\RR \psi_R(s) (\F h_{n,\alpha})(s)(\F\inv\eta)(s)\,\dd s,
\end{align*}
where $\xi_R(s)=\smash{\int_\RR}\phi_R(\tau-s)\eta(\tau)\,\dd\tau$, $s\in\RR$. Since $\s(A)\cap i\RR=\emptyset$ the resolvent of $A$ extends holomorphically across the imaginary axis and hence is uniformly bounded in an open neighbourhood of $i\supp\psi_R$. It follows from \eqref{eq:FT} and another application of the dominated convergence theorem that 
$$\phi_R*h_n=\F\inv(\psi_R m_{n}\,\F h),$$
 where $m_{n}(s)=n! R(is,A)^nA\inv$ and $h(s)=g_0(s)T(s)x$, $s\in\RR$. A straightforward estimate using Plancherel's theorem now gives 
$$\|\phi_R*h_n\|_{L^2(\RR,X)}\le \|\psi_Rm_n\|_{L^\infty(\RR,\B(X))}\|h\|_{L^2(\RR,X)}.$$
Note that $\smash{\|h\|_{L^2(\RR,X)}}\le Kt^{1/2}\|x\|$. Moreover,  
$$isR(is,A)^nA\inv x=R(is,A)^{n-1}A\inv x+R(is,A)^nx,\quad s\in\RR,$$
and hence $|s|\|R(is,A)^nA\inv\|\lesssim M(|s|)^{n-1}+M(|s|)^n,$ $s\in\RR$. By rescaling $M$ if necessary we may assume that $M(s)\ge1$ for all $s\ge0$, and then
$$\|R(is,A)^nA\inv\|\lesssim\frac{M(|s|)^n}{\max\{s_0,|s|\}},\quad s\in\RR,$$
where $s_0>0$ is fixed but arbitrary. Now since $M$ is non-decreasing and has positive increase there exist constants $\alpha>0$ and $c\in(0,1]$ such that 
$$\frac{M(R)}{M(|s|)}\ge c\left(\frac{R}{|s|}\right)^\alpha,\quad  R\ge|s|\ge s_0.$$
We now make a specific choice of $n$ by setting $n=\lceil\alpha\inv\rceil$. A simple calculation then gives
$$\|\psi_Rm_n\|_{L^\infty(\RR,\B(X))}\lesssim n!\sup_{|s|\le R}\frac{M(|s|)^n}{\max\{s_0,|s|\}}\le \frac{n!}{R}\left(\frac{M(R)}{c}\right)^n\!,\quad R\ge s_0.$$ Combining the above estimates in \eqref{eq:Holder} shows that for $R\ge s_0$ we have
\begin{equation}\label{eq:rem_est}
\left\|\frac{n+1}{t^{n+1}}\int_0^tT(t-s)\phi_R*h_n(s)\,\dd s\right\|\lesssim  (n+1)!\frac{\|x\|}{R}\left(\frac{M(R)}{ct}\right)^n\!,
\end{equation}
where the implicit constant is independent of $R$, $t$ and $x$. Using \eqref{eq:rem_est} in \eqref{eq:sg_int} along with our earlier estimate  gives
\begin{equation}\label{eq:final_est}
\|T(t)A\inv\|\lesssim\frac{1}{R}\left(P_n(Rt)+\frac{n+1}{t}P_{n-1}(Rt)+(n+1)!\left(\frac{M(R)}{ct}\right)^n\right)
\end{equation}
for all $R\ge s_0$ and $t>0$, where the implicit constant is independent of both $R$ and $t$.   In fact, the implicit constant would also be independent of $n$ if it were still free to vary, and this will become important in the proof of Theorem~\ref{thm:quasi} below. We now set $R=M\inv(ct)$ for $t\ge c\inv M(s_0)$. Then the first two terms in \eqref{eq:final_est} are uniformly bounded because the functions $P_n$, $P_{n-1}$ defined in \eqref{eq:Pm} are non-increasing, and the final term is constant by our choice of $R$. Hence the result follows from Proposition~\ref{prp:M_inv}.
\end{proof}

\begin{rem}\label{rem:DC}
The techniques used in the above proof can be adapted and combined with ideas from \cite{ChiSei16} to give an alternative proof of Theorem~\ref{thm:BD}. In this case the number $n$ is allowed to grow arbitrarily large and one needs to control the norms  $\|\Phi_k\|_{L^1}$, $k\in\NN$, by appealing to the Denjoy-Carleman theorem \cite[Theorem~1.3.8]{Hoe90}; see the proof of Theorem~\ref{thm:quasi} below. Note also that in the general Banach space setting  Plancherel's theorem has to be replaced by cruder ways of estimating the norms of Fourier transforms.
\end{rem}

The conclusion of Theorem~\ref{thm:inf} becomes false if we drop the assumption of positive increase. In fact, it is easy to construct examples of bounded normal semigroups $\T$ whose generator $A$ satisfies $\s(A)\cap i\RR=\emptyset$ and $\|R(i s,A)\|\le 1+\log|s|$, $|s|\ge1$, but for which 
\begin{equation}\label{eq:normal_ex}
\|T(t)A\inv\|\sim\exp\big(-2t^{1/2}\big),\quad t\to\infty,
\end{equation}
so that \eqref{eq:Minv} is violated; see \cite[Example~5.2]{BaChTo16}. We shall see in Section~\ref{sec:quasi} that this is a special case of a much more general result, Theorem~\ref{thm:normal_rate}, which allows us to compute the precise rate of decay for normal semigroups with arbitrary resolvent growth. On the other hand, if we let $M\colon[0,\infty)\to(0,\infty)$ be defined by $M(s)=\smash{\sup_{|r|\le s}}\|R(ir,A)\|$, $s\ge0$, then even for semigroups of contractions on a Hilbert space it is possible for \eqref{eq:Mlog} to hold with $\smash{\Mlog\inv}$ replaced by $M\inv$ despite $M$ not having positive increase.  Indeed, let us consider the contraction semigroup $\T$ generated by $A=B-I$, where $B$ is the generator considered in \cite[Example~5.1.10]{ABHN11}. Thus $B$ is an infinite direct sum of Jordan blocks of increasing size. This is a modification of the well-known example due to Zabczyk \cite{Zab75} showing that even for semigroups on Hilbert spaces the spectral bound can be strictly smaller than the growth bound. It is straightforward to show that $\|T(t)A\inv\|=O(e^{-t/2})$ as $t\to\infty$. Using the estimate in \eqref{eq:lb} and sharpening the lower bound for the resolvent obtained in \cite[Example~5.1.10]{ABHN11} we see that $M(s)\asymp\log s$ as $s\to\infty$. Thus \eqref{eq:Mlog} holds for some $c>0$ with $\smash{\Mlog\inv}$ replaced by $M\inv$ but $M$ does not have positive increase. 

One crucial feature of the previous example is that $M$ is unbounded even though $\dist(is,\s(A))\ge1$ for all $s\in\RR$. As we shall see now, the situation changes if we restrict attention to cases in which the resolvent growth is controlled by the distance to the spectrum. Indeed, the following result is similar to \cite[Proposition~5.1]{BaChTo16} and shows for a large class of semigroups, including in particular all normal semigroups, that the assumption of positive increase is in fact \emph{necessary} for \eqref{eq:Minv} to hold, so Theorem~\ref{thm:BD} is optimal in this sense. Note that the assumptions made in our result appear to be weaker, and are certainly easier to verify, than those of \cite[Proposition~5.1]{BaChTo16}. We shall take advantage of this in Section~\ref{sec:wave} below. 

\begin{thm}\label{thm:inf_nec}
Let $X$ be a Banach space and let $A$ be the generator of a $C_0$-semigroup $\T$ on $X$. Suppose that $\s(A)\subseteq\CC_-$ and that $M\colon\RR_+\to(0,\infty)$ is a continuous non-decreasing function such that $M(s)\to\infty$ as $s\to\infty$ and 
\begin{equation}\label{eq:spectral_ass}
\delta M(s)\le \sup_{|r|\le s}\frac{1}{\dist(ir,\s(A))}\le \sup_{|r|\le s}\|R(ir,A)\|\le M(s),\quad s\ge0,
\end{equation}
for some constant $\delta\in(0,1]$. If 
\begin{equation}\label{eq:Minv_ass}
\|T(t)A\inv\|=O\left(M\inv(ct)\inv\right),\quad t\to\infty,
\end{equation}
for some $c>0$ then $M$ has positive increase.
\end{thm}

\begin{proof}
Consider the function $N\colon\RR_+\to(0,\infty)$ given by 
$$N(s)=\sup_{|r|\le s}\frac{1}{\dist(ir,\s(A))},\quad s\ge0.$$
Then $\delta M(s)\le N(s)\le M(s)$, $s\ge0$. Recall that the spectral radius of a bounded linear operator is always dominated by the norm of the operator. Hence by \eqref{eq:Minv_ass} and the spectral inclusion theorem for the Hille-Phillips functional calculus \cite[Theorem~16.3.5]{HiPh74} there exists a constant $C>0$ such that if $\alpha+i\beta\in\s(A)$ then 
$$\frac{e^{\alpha t}}{|\alpha+i\beta|}\le\|T(t)A\inv \|\le\frac{C}{M\inv(ct)}$$
for all sufficiently large $t$. It follows that
\begin{equation}\label{eq:at_est}
-\alpha t\ge\log\left(\frac{N\inv (\delta ct)}{C|\alpha+i\beta|}\right)
\end{equation}
whenever $\alpha+i\beta\in\s(A)$ and $t>0$ is sufficiently large. Now given $s\ge0$ we may find $r\in[-s,s]$ and $\alpha+i\beta\in\s(A)$ such that $N(s)=|\alpha+i\beta-ir|\inv$. Note that $-\alpha\le  N(s)\inv$ and that, for $s$ sufficiently large, we have $|\alpha+i\beta|\le 2s$. In fact, one could replace the factor $2$ by $(1-\ep)\inv$ for any $\ep\in(0,1)$ here, and we shall make use of this fact in the proof of Theorem~\ref{thm:normal_rate} below. Let $\lambda\ge1$ and, for $s$ sufficiently large, let $t=(\delta c)\inv N(\lambda s)$. Then \eqref{eq:at_est} yields
$$\frac{N(\lambda s)}{N(s)}\ge\delta c\log\left(\frac{\lambda}{2C}\right),$$
and replacing $\delta$ by $\delta^2$ we see that the same estimate holds with $N$ replaced by $M$. Hence $M$ has positive increase by Lemma~\ref{lem:pos_inc}, as required.
\end{proof}

\subsection{Singularity at zero}\label{sec:zero}

Let $X$ be a Banach space and let $A$ be the generator of a  $C_0$-semigroup $\T$ on $X$. It is desirable to have at one's disposal a version of Theorem~\ref{thm:inf} which applies when $\sigma(A)\cap i\RR$ is non-empty. In the simplest yet most important case we have $\sigma(A)\cap i\RR=\{0\}$, and this situation arises in a number applications including various problems where the solutions of  \eqref{eq:ACP_int} converge to some non-zero steady-state; see for instance \cite{LoMK17, MKSe18, PauSei17, Sta17b, Sta17}. In Section~\ref{sec:zero_inf} below we shall allow for the resolvent norms $\|R(is,A)\|$ to be unbounded not only as $|s|\to0$ but also as $|s|\to\infty$, or in other words we allow for singularities of the resolvent along the imaginary axis both at zero and at infinity. For now, however, we shall make an additional assumption on our semigroup which rules out a singularity at infinity. Recall from \cite{BaSr03} the definition of the {\em non-analytic growth bound} $\zeta (T)$ of a $C_0$-semigroup $\T$ on $X$, namely
$$\zeta (T)  = \inf \Big\{ \omega\in\RR :  \sup_{t>0}e^{-\omega t}\| T(t)-S(t)\|<\infty\;\mbox{for some $S\in\mathcal{H}(\B(X))$}\Big\},$$
where $\mathcal{H}(\B(X))$ denotes the set of all maps $S\colon(0,\infty)\to\B(X)$ which have an exponentially bounded analytic extension to some sector containing $(0,\infty)$. It follows from properties of the Laplace transform of analytic functions that if $\zeta (T) <0$, then $\s(A)\cap i\RR$ is a compact set and $\smash{\sup_{|s|\ge s_0}}\|R(i s,A)\|<\infty$
whenever $s_0\ge0$ is sufficiently large. For bounded $C_0$-semigroups on Hilbert spaces these conditions are even equivalent to having $\zeta (T)<0$; see \cite{BaSr03} for a proof of this fact using the theory of Fourier multipliers. The following result is proved in \cite{ChiSei16}.

\begin{thm}\label{thm:CS}
Let $X$ be a Banach space and let $A$ be the generator of a bounded $C_0$-semigroup $\T$ on $X$ with $\zeta(T)<0$. Suppose that $\s(A)\cap i\RR=\{0\}$ and that $M\colon[1,\infty)\to(0,\infty)$ is a continuous non-decreasing function such that $\|R(is\inv,A)\|\le M(|s|)$, $|s| \ge1$. Then there exists a constant $c>0$ such that
\begin{equation}\label{eq:mlog}
\|T(t)AR(1,A)\|=O\big(\Mlog\inv(ct)\inv\big),\quad t\to\infty,
\end{equation}
where $\Mlog\colon[1,\infty)\to(0,\infty)$ is defined by $\Mlog(s)=M(s)(\log s+\log(1+M(s)))$, $s\ge1$. 
\end{thm}

It is shown in \cite[Theorem~6.10]{BaChTo16} that if $\|T(t)AR(1,A)\|\to0$ as $t\to\infty$ then necessarily $\s(A)\cap i\RR\subseteq\{0\}$ and $\sup_{|s|\ge1}\|R(is,A)\|<\infty$, so the spectral assumption and the condition on the non-analytic growth bound made in Theorem~\ref{thm:CS} are natural, especially when $X$ is a Hilbert space; see also \cite[Section~4.2]{Sei15}. Moreover, by \cite[Corollary~6.11]{BaChTo16} we see that in the setting of Theorem~\ref{thm:CS} for the choice of $M\colon[1,\infty)\to(0,\infty)$ given by $M(s)=\smash{ \sup_{ s\inv\le |r|\le 1}} \|R(ir,A)\|$, $s\ge1$, there exist constants $C,c>0$ such that 
$$\|T(t)AR(1,A)\|\ge\frac{c}{M\inv(Ct)}$$
for all sufficiently large $t$, at least provided $\|R(is,A)\|$ grows faster than $|s|\inv$ as $|s|\to0$. It is further shown in \cite{BaChTo16} that  if $X$ is a Hilbert space then one may replace $\smash{\Mlog\inv}$ by $M\inv$ in \eqref{eq:mlog}  when $M(s)=C  s^\alpha$, $s\ge1$, for some constants $C>0$, $\alpha\ge1$, and also if $M$ is a regularly varying function of positive index satisfying certain supplementary conditions. Our next result, which is an analogue of Theorem~\ref{thm:inf}, extends these statements. As discussed above, the Hilbert space setting allows to replace the negativity assumption on the non-analytic growth bound appearing in Theorem~\ref{thm:CS} by a more tractable boundedness condition on the resolvent norms.

\begin{thm}\label{thm:zero}
Let $X$ be a Hilbert space and let $A$ be the generator of a bounded $C_0$-semigroup $\T$ on $X$. Suppose that $\s(A)\cap i\RR=\{0\}$, that $\smash{\sup_{|s|\ge1}}\|R(is,A)\|<\infty$ and that $M\colon[1,\infty)\to(0,\infty)$ is a continuous non-decreasing function of positive increase such that 
$\|R(is\inv,A)\|\le M(|s|)$, $|s| \ge1$. Then 
\begin{equation}\label{eq:minv}
\|T(t)AR(1,A)\|=O\left(M\inv(t)\inv\right),\quad t\to\infty.
\end{equation}
\end{thm}

\begin{proof}
The proof is similar to that of Theorem~\ref{thm:inf}. Let $\psi\colon\RR\to\CC$ be a Schwartz function such that $\|\psi\|_{L^\infty}=1$ and $\psi(s)=1$ for $|s|\le 1,$ and let $\phi=\F\inv\psi$. Temporarily fix $x\in X$, $n\in\NN$ and $t>0$, and  define the map $h_n\colon\RR\to X$ by $h_n(s)=g_n(s)T(s)AR(1,A)x$, $s\in\RR$, where the semigroup is extended by zero to the whole real line and where $g_n$ is as defined in \eqref{eq:g}. Moreover, let $H_n\colon\RR\to X$ be given by 
$$H_n(s)=\int_0^sh_n(\tau)\,\dd\tau,\quad s\in\RR.$$
In particular $H_n(s)=0$ for $s<0$, and using integration by parts we obtain
$$\|H_n(s)\|\le 2 K s^n \|R(1,A)\| \|x\|,\quad s\ge0,$$
where $K=\sup_{t\ge0}\|T(t)\|$. For $r\in(0,1]$ we let $\phi_r(t)=r\phi(rt)$, $t\in\RR$, and $\psi_r=\F(\phi_{r})$, as in the proof of Theorem~\ref{thm:inf}. Integration by parts gives
$$\phi_r*h_n(s)=r\int_0^\infty\phi'(rs-\tau)H_n(r\inv\tau)\,\dd\tau,\quad s\in\RR,$$
and hence
$$\|\phi_r*h_n(s)\|\lesssim r\|x\|\int_0^\infty|\phi'(rs-\tau)|\left(\frac{\tau}{r}\right)^n\,\dd\tau,\quad s\in\RR.$$
As in the proof of Theorem~\ref{thm:inf} we now introduce functions $\Phi_k\colon\RR\to\RR$, $k\in\NN$, defined as in \eqref{eq:Phi}  but with $\Phi_1=|\phi'|$. This leads to the estimate
\begin{equation}\label{eq:diff_int_est_zero}
\left\|\frac{n+1}{t^{n+1}}\int_0^tT(t-s)\phi_r*h_n(s)\,\dd s\right\|\lesssim r\|x\|P_n(rt),
\end{equation}
where the implicit constant is independent of $r$, $n$, $t$ and $x$, and where $P_{n}$ is as defined in \eqref{eq:Pm}.

By our assumption that $\smash{\sup_{|s|\ge1}}\|R(is,A)\|<\infty$ and a standard Neumann series argument there exists $\ep>0$ such that $\|R(z,A)\|$ is uniformly bounded over all $z\in\CC$ satisfying $\dist(z,i\supp(1-\psi_r))<\ep$. Hence as in the proof of Theorem~\ref{thm:inf} we have
$$(\delta-\phi_r)*h_n=\F\inv\big((1-\psi_r) m_{n}\,\F h\big),$$
where $m_n(s)=n!AR(is,A)^n$, $s\in\RR\setminus\{0\}$, and $h(s)=g_0(s)T(s)R(1,A)x$, $s\in\RR$. Using the fact that $M(s)\ge s$, $s\ge1$, it is straightforward to show that 
$\|AR(is,A)^n\|\le2|s|M(|s|\inv)^n,$ $0<|s|\le1.$ By rescaling $M$ if necessary we may assume that $\|R(is,A)\|\le M(1)$, $|s|\ge1$. Since $M$ is assumed to have positive increase it follows as before that for an appropriate choice of $n$ we have 
$$\|(1-\psi_r)m_n\|_{L^\infty(\RR,\B(X))}\lesssim r\left(\frac{M(r\inv)}{c}\right)^n\!,$$
where $c>0$ is a constant. We deduce, upon applying Plancherel's theorem and H\"older's inequality, that for sufficiently small values of $r$ we have
$$\left\|\frac{n+1}{t^{n+1}}\int_0^tT(t-s)(\delta-\phi_r)*h_n(s)\,\dd s\right\|\lesssim r\|x\|\left(\frac{M(r\inv)}{ct}\right)^n\!,$$
where the implicit constant is independent of $r$, $t$ and $x$. Combining this with \eqref{eq:diff_int_est_zero} as in the proof of Theorem~\ref{thm:inf} gives
$$\|T(t)AR(1,A)\|\lesssim r\left(P_n(rt)+\left(\frac{M(r\inv)}{ct}\right)^n\right),$$
where the implicit constant is independent of both $r$ and $t$. We now set $r=M\inv(ct)\inv$ for sufficiently large $t$. Then in particular $rt\ge c\inv$, and since $P_n$ is non-increasing the result follows from Proposition~\ref{prp:M_inv}.
\end{proof}

As in Section~\ref{sec:inf} we can show that the condition of positive increase is not only sufficient but even necessary for the conclusion of Theorem~\ref{thm:zero} to hold, at least for a large class of semigroups. We omit the proof, which is similar to that of Theorem~\ref{thm:inf_nec}; see also \cite[Proposition~6.13]{BaChTo16}.

\begin{thm}\label{thm:zero_nec}
Let $X$ be a Banach space and let $A$ be the generator of a $C_0$-semigroup $\T$ on $X$. Suppose that $0\in\s(A) \subseteq\CC_-\cup\{0\}$ and that  $M\colon[1,\infty)\to(0,\infty)$ is a continuous non-decreasing function such that 
$$\delta M(s)\le\!\sup_{s\inv\le|r|\le 1}\frac{1}{\dist(ir,\s(A))}\le \!\sup_{s\inv\le|r|\le 1}\|R(ir,A)\|\le  M(s),\quad s\ge1,$$ 
for some constant $\delta\in(0,1]$. If 
$$\|T(t)AR(1,A)\|=O\left(M\inv(ct)\inv\right),\quad t\to\infty,$$
for some $c>0$ then $M$ has positive increase.
\end{thm}

\subsection{Singularities at zero and infinity}\label{sec:zero_inf} 

We now consider the remaining case where the resolvent operator has  singularities at both zero and infinity. The following result is proved in \cite{Ma11}.

\begin{thm}\label{thm:Mar}
Let $X$ be a Banach space and let $A$ be the generator of a bounded $C_0$-semigroup $\T$ on $X$. Suppose  that $\s(A)\cap i\RR=\{0\}$ and that $M_0,M_\infty\colon[1,\infty)\to(0,\infty)$ are continuous non-decreasing functions such that $\|R(is\inv,A)\|\le M_0(|s|)$ and $\|R(is,A)\|\le M_\infty(|s|)$, $|s| \ge1$. Let $M\colon[1,\infty)\to(0,\infty)$ be defined by $M(s)=\max\{M_0(s),M_\infty(s)\}$, $s\ge1$. Then there exists a constant $c>0$ such that
\begin{equation}\label{eq:Mlog2}
\|T(t)AR(1,A)^2\|=O\big(\Mlog\inv(ct)\inv\big),\quad t\to\infty,
\end{equation}
where $\Mlog\colon[1,\infty)\to(0,\infty)$ is defined by $\Mlog(s)=M(s)(\log s+\log(1+M(s)))$, $s\ge1$. 
\end{thm}

The spectral assumption is again natural here, since by \cite[Corollary~6.2]{BaChTo16} we have $\s(A)\cap i\RR\subseteq\{0\}$ whenever $\|T(t)AR(1,A)^2\|\to0$ as $t\to\infty$. Note also that if $X$ is a Hilbert space and the function $M_\infty$ is bounded then $\zeta(T)<0$ and hence the conclusion of Theorem~\ref{thm:Mar} follows from Theorem~\ref{thm:CS} in this case. It is shown in \cite[Corollary~8.2]{BaChTo16} that in the setting of Theorem~\ref{thm:Mar} for the smallest possible choices of $M_0$, and $M_\infty$, defined as in Sections \ref{sec:inf} and \ref{sec:zero}, there exist constants $C,c>0$ such that 
$$\|T(t)AR(1,A)^2\|\ge\frac{c}{M\inv(Ct)}$$
for all sufficiently large $t$, at least provided $\|R(is,A)\|$ grows faster than $|s|\inv$ as $|s|\to0$. It is further shown in \cite[Theorem~8.4]{BaChTo16} that  if $X$ is a Hilbert space then one may replace $\smash{\Mlog\inv}$ by $M\inv$ in \eqref{eq:Mlog2}  when  $M_0(s)=Cs^\alpha$ and $M_\infty(s)=c s^\beta$, $s\ge1$, for some  constants $C, c,\beta>0$ and $\alpha\ge1$. However, the techniques used in \cite{BaChTo16} do not allow the authors to obtain similar results for any broader class of functions. Our next result shows that one may replace $\smash{\Mlog\inv}$ by $M\inv$ in \eqref{eq:Mlog2} whenever $M$ has positive increase.

\begin{thm}\label{thm:zero_inf}
Let $X$ be a Hilbert space and let $A$ be the generator of a bounded $C_0$-semigroup $\T$ on $X$. Suppose that $\s(A)\cap i\RR=\{0\}$ and that $M_0,M_\infty\colon[1,\infty)\to(0,\infty)$ are continuous non-decreasing functions such that 
$\|R(is\inv,A)\|\le M_0(|s|)$ and $\|R(is,A)\|\le M_\infty(|s|)$, $|s| \ge1$. Let $M\colon[1,\infty)\to(0,\infty)$ be defined by $M(s)=\max\{M_0(s),M_\infty(s)\}$, $s\ge1$, and suppose that $M$ has positive increase. Then 
$$\|T(t)AR(1,A)^2\|=O\left(M\inv(t)\inv\right),\quad t\to\infty.$$
\end{thm}

\begin{proof}
The proof follows the same pattern as those of Theorems \ref{thm:inf} and \ref{thm:zero}, and indeed combines ideas from both proofs. This time the splitting arises from the decomposition 
$$\delta=(\delta-\phi_R)+(\phi_R-\phi)+(\phi-\varphi_r)+\varphi_r,$$
where $r\in(0,1]$, $R>0$ and the notation is as before, with $\phi$ being the same as in the proof of Theorem~\ref{thm:inf} and $\varphi$ being the function arising in the proof of Theorem~\ref{thm:zero}. The integrals corresponding to the first two terms of the splitting can now be dealt with as in the proof of Theorem~\ref{thm:inf}, the terms arising from the second two as in the proof of Theorem~\ref{thm:zero}. 
\end{proof}

Once again it can be shown that the condition of positive increase is necessary for the conclusion of Theorem~\ref{thm:zero_inf} to hold, at least for a large class of semigroups. The proof involves no new ideas, so as in the case of Theorem~\ref{thm:zero_nec} we omit it.

\begin{thm}\label{thm:zero_inf_nec}
Let $X$ be a Banach space and let $A$ be the generator of a $C_0$-semigroup $\T$ on $X$. Suppose that $0\in\s(A)\subseteq\CC_-\cup\{0\}$ and that $M_0,M_\infty\colon[1,\infty)\to(0,\infty)$ are continuous non-decreasing functions such that 
$$\delta M_0(s)\le\!\sup_{s\inv\le|r|\le 1}\frac{1}{\dist(ir,\s(A))}\le \!\sup_{s\inv\le|r|\le 1}\|R(ir,A)\|\le  M_0(s),\quad s\ge1,$$ 
and 
$$\delta M_\infty(s)\le\sup_{1\le|r|\le s}\frac{1}{\dist(ir,\s(A))}\le \sup_{1\le|r|\le s}\|R(ir,A)\|\le  M_\infty(s),\quad s\ge1,$$ 
for some constant $\delta\in(0,1]$. Let $M\colon[1,\infty)\to(0,\infty)$ be defined by $M(s)=\max\{M_0(s),M_\infty(s)\}$, $s\ge1$. If 
$$\|T(t)AR(1,A)^2\|=O\left(M\inv(ct)\inv\right),\quad t\to\infty,$$
for some $c>0$ then $M$ has positive increase.
\end{thm}

\section{Decay for resolvent growth with quasi-positive increase}\label{sec:quasi}

The purpose of this section is to investigate rates of decay in the case of resolvent growth which does not have positive increase. 
Our main interest is in the case where the resolvent growth is sub-polynomial. Since this situation cannot arise when there is a singularity at zero, it is natural to consider only the case of a singularity at infinity. We begin by extending the terminology introduced in Section~\ref{sec:functions}. Given measurable functions $M,N\colon\RR_+\to(0,\infty)$ with $N$ non-decreasing we say that $M$ has \emph{quasi-positive increase (with auxiliary function $N$)} if there exist constants  $c\in(0,1]$ and $s_0>0$ such that
\begin{align}\label{eq:quasi-pos_inc}
 \frac{M(\lambda s)}{M(s)}\ge c\lambda^{1/N(\lambda s)},\quad \lambda\ge1, \,s\ge s_0.
\end{align}
In particular, a measurable function $M\colon\RR_+\to(0,\infty)$ has positive increase if and only if it has quasi-positive increase and admits a bounded auxiliary function. Suppose, for instance, that $M\colon\RR_+\to(0,\infty)$ is a slowly varying function which admits a representation as in \eqref{eq:slow_rep} for some $a>0$, with $p$ positive, continuous and non-increasing and with $q$ constant. We shall refer to such slowly varying functions as being \emph{normalised}. It is then straightforward to verify that \eqref{eq:quasi-pos_inc} is satisfied for the function $N(s)=p(s)\inv$, $s\ge s_0$, if we choose $c=1$ and any $s_0\ge a$. Furthermore,   \emph{any} non-decreasing function $M\colon\RR_+\to(0,\infty)$ has quasi-positive increase with auxiliary function $N(s)=\log(2+s)$, $s\ge0$. In this case \eqref{eq:quasi-pos_inc} holds for $c=e\inv$ and $s_0=1$.

Recall that Theorem~\ref{thm:inf} becomes false if we drop the assumption of positive increase. The following result is a generalisation of Theorem~\ref{thm:inf} to the case where the resolvent growth has quasi-positive increase. Here and in the remainder of this section, given two functions $M\colon\RR_+\to(0,\infty)$  and $K\colon[a,\infty)\to(0,\infty)$ for some $a\ge0$ we shall let $M_K\colon[a,\infty)\to(0,\infty)$ denote the function defined by $M_K(s)=M(s)K(s),$ $s\ge a,$ even though strictly speaking this is inconsistent with the notation $\Mlog$ used elsewhere in the paper.

\begin{thm}\label{thm:quasi}
Let $X$ be a Hilbert space and let $A$ be the generator of a bounded $C_0$-semigroup $\T$ on $X$ with $\s(A)\cap i\RR=\emptyset$. Let $M,N\colon\RR_+\to(0,\infty)$ be continuous non-decreasing functions and suppose that $M(s)\to\infty$ as $s\to\infty$, that $N(s)=O(\log s)$ as $s\to\infty$ and that $M$ has quasi-positive increase with auxiliary function $N$. Suppose further that $\|R(is,A)\|\le M(|s|)$, $s\in\RR$. Then 
\begin{equation}\label{eq:MNinv}
\|T(t)A\inv\|=O\left(M_K\inv(c et)\inv\right),\quad t\to\infty,
\end{equation}
where  $c$ is as in \eqref{eq:quasi-pos_inc}  and  
$$K(s)=N(s)\left(1+\frac{3\log N(s)}{2N(s)}\right),\quad s\ge N\inv(1).$$
In particular, given any $\ep\in(0,1)$ we have
\begin{equation}\label{eq:MNinv_ep}
\|T(t)A\inv\|=O\Big(M_{N}\inv\big(ce(1-\ep)t)\big)\inv\Big),\quad t\to\infty.
\end{equation}
\end{thm}

\begin{proof}
If $N$ is bounded then $M$ has positive increase and  the result follows from Theorem~\ref{thm:inf}, so we may assume that $N(s)\to\infty$ as $s\to\infty$. Let us first prove \eqref{eq:MNinv}. Note that by Stirling's formula 
\begin{equation}\label{eq:Stir}
(n+1)!\asymp \frac{n^n}{e^n} \Big(1+\frac{3\log n}{2n}\Big)^n,\quad n\to\infty.
\end{equation}
We use the same notation as in the proof of Theorem~\ref{thm:inf} and proceed in exactly the same way except that we now allow our choice of $n$ to be depend on $R$. Indeed, if we choose $n=\lceil N(R)\rceil$ then \eqref{eq:final_est} and \eqref{eq:Stir} imply that for $R$  sufficiently large and $t>0$ we have
\begin{equation}\label{eq:quasi_proof}
 \|T(t)A\inv\|\lesssim \frac{1}{R}\left(P_n(Rt)+\frac{n+1}{t}P_{n-1}(Rt)+\left(\frac{M_K(R)}{c e t}\right)^n\right),
\end{equation}
where  the implicit constant is independent of both $R$ and $t$. We now set $R=M_K\inv(ce t)$ for $t$ sufficiently large. Thus  \eqref{eq:MNinv} follows provided the first two terms inside the brackets remain uniformly bounded as $t \to\infty$.  By the Denjoy-Carleman theorem  \cite[Theorem~1.3.8]{Hoe90} we may assume that the function $\psi$ in addition to the properties already mentioned satisfies $\|\psi^{(k)}\|_{L^\infty}\le C_k$, where $C_k=\smash{ C^kk^{2k}}$, $k\in\ZZ_+,$ for some constant $C>0$.  Integrating by parts we then find that $|\phi(s)|\lesssim C_k(1+|s|)^{-k}$ for all $k\in\ZZ_+$ and $s\in\RR$, and hence  $\|\Phi_k\|_{L^1}\lesssim C_{k+2}$ for all $k\in\ZZ_+$. Using \eqref{eq:Pm} and estimating crudely we thus find, after adjusting the value of the constant $C$, that for $t\ge1$ we have
$$P_n(Rt)\lesssim C_3+\sum_{k=1}^\infty R\inv\big(C(N(R)+3)^3\big)^{k+3},$$
where the implicit constant is independent of $t$ and hence of $R$. Since $N$ grows at most logarithmically, we deduce that $P_n(Rt)$ is uniformly bounded as $t\to\infty$. Moreover, since $ N(R)\lesssim t$ we see similarly that the second term in \eqref{eq:quasi_proof} remains bounded as $t$ grows large. This completes the proof of \eqref{eq:MNinv}. 
In order to obtain \eqref{eq:MNinv_ep} it suffices to observe that given any $\ep\in(0,1)$ we have $M_K(s)\le (1-\ep)\inv M_{N}(s)$ for all sufficiently large values of $s$.
\end{proof}

\begin{rem}\label{rem:simple}
\begin{enumerate}[(a)]
\item In concrete situations the estimate in \eqref{eq:MNinv} can be difficult to apply since it requires inverting the function $M_K$. However, with a minor change in the proof of Theorem~\ref{thm:quasi} it is possible to obtain an alternative estimate which involves only the inverse of the simpler function $M_N$. Indeed, if we replace \eqref{eq:Stir} by the estimate $(n+1)!\asymp e^{-n}n^{n+3/2}$, $n\to\infty$, then choosing $R=\smash{M_N\inv}(cet)$ for sufficiently large $t$ gives
\begin{equation}\label{eq:simple}
\|T(t)A\inv\|=O\bigg(\frac{N(M_N\inv(cet))^{3/2}}{M_N\inv(cet)}\bigg),\quad t\to\infty.
\end{equation}
In general, one would expect this estimate to be significantly better than \eqref{eq:MNinv_ep} but perhaps not quite as sharp as \eqref{eq:MNinv}. As we shall see shortly, in some important cases \eqref{eq:MNinv} and \eqref{eq:simple} lead to the same rate of decay.
\item Note that for the `universal' auxiliary function $N(s)=\log(2+s),$ $s\ge0$, the estimate in \eqref{eq:MNinv_ep} implies that $\|T(t) A\inv\|=O(\smash{M_N\inv}(ct)\inv)$, $t\to\infty$, for all $c\in (0,1)$, which is an improvement over \eqref{eq:Mlog} in this case.
\end{enumerate}
\end{rem}

The assumptions made in Theorem~\ref{thm:quasi} are natural. Indeed, since $M$ is assumed to be non-decreasing the growth assumption on $N$ in view of the comments made at the beginning of this section involves no essential loss of generality. Moreover, if $N$ were allowed to grow faster than logarithmically then $M_N$ and $M_K$ would in general grow faster than the function $\Mlog$ appearing in Theorem~\ref{thm:BD}, so \eqref{eq:Mlog} would give a better estimate than Theorem~\ref{thm:quasi}. Finally, the assumption that $M$ is unbounded, which in Section~\ref{sec:opt_dec} was implicit in the assumption of positive increase,  is also natural here. Indeed, if $\T$ is a bounded $C_0$-semigroup on a Hilbert space whose generator has uniformly bounded resolvent along the imaginary axis then $\T$ is in fact uniformly exponentially stable by the Gearhart-Pr\"uss theorem \cite[Theorem~5.2.1]{ABHN11}, whereas the starting point for this line of research  in a sense is the \emph{absence} of uniform stability.

When applying Theorem~\ref{thm:quasi} to specific functions $M$ which have quasi-positive increase, one has some freedom in choosing a suitable function $N$ and suitable constants $c$ and $s_0$ for which \eqref{eq:quasi-pos_inc} is satisfied. For the best rate of decay in \eqref{eq:MNinv_ep} one would like to choose the smallest possible function $N$ and the largest possible constant $c$. More precisely, if $N_c$ is the smallest function such that \eqref{eq:quasi-pos_inc} holds for $c\in(0,1]$ and some $s_0>0$ then the best possible choice of $N$ is obtained by requiring $c\inv N_c(s)$ to be as small as possible for large values of $s$. Note that a non-decreasing function $N\colon\RR_+\to(0,\infty)$ satisfies \eqref{eq:quasi-pos_inc} if and only if
\begin{equation}\label{eq:N_opt}
N(s)\ge\sup_{1<\lambda\le\frac{s}{s_0}}\frac{\log\lambda}{\log\Big(\frac{M(s)}{cM(\lambda\inv s)}\Big)},\quad s>s_0.
\end{equation}
In particular, for the class of normalised slowly varying functions considered at the beginning of this section, it follows that the choice $N(s)=p(s)\inv$, $s\ge s_0$, in fact gives the smallest auxiliary function satisfying \eqref{eq:quasi-pos_inc} for $c=1$.

\begin{ex}\phantomsection\label{ex:MN}
\begin{enumerate}[(a)]
\item\label{ex:log^a} Let $\alpha>0$ and define $M\colon\RR_+\to(0,\infty)$ by $M(s)=1$ for $s\in[0,e)$ and $M(s)=\log (s)^\alpha$ for $s\ge e$. From \eqref{eq:N_opt} it is easy to see that any pointwise minimal auxiliary function must eventually  be proportional to the logarithm function, and then straightforward optimisation arguments lead to the choice $N(s)=(1+\alpha)\inv\log s$ for sufficiently large values of $s$, which satisfies \eqref{eq:quasi-pos_inc} for $c=e\inv(\alpha\inv+1)^\alpha$. Hence given $\ep\in(0,1)$ it follows from \eqref{eq:MNinv_ep} that 
\begin{equation}\label{eq:ep_loss}
\|T(t)A\inv\|=O\left(\exp\left(-\big(c_\alpha(1-\ep)t\big)^\frac{1}{1+\alpha}\right)\right),\quad t\to\infty,
\end{equation}
where $c_\alpha=\alpha^{-\alpha}(1+\alpha)^{1+\alpha}$. Using either \eqref{eq:MNinv} or \eqref{eq:simple} we obtain the significantly sharper estimate 
\begin{equation}\label{eq:refined}
 \|T(t)A\inv\| =O\left( t^{\frac{3}{2(\alpha+1)}}\exp\left(-(c_\alpha t)^\frac{1}{1+\alpha}\right)\right),\quad t\to\infty.
\end{equation}
\item\label{ex:explog} Let $\alpha\in(0,1)$ and define $M\colon\RR_+\to(0,\infty)$ by $M(s)=1$ for $s\in[0,1)$ and $M(s)=\exp(\log (s)^\alpha)$ for $s\geq 1$. It follows from \eqref{eq:N_opt} that the optimal choice of $N$ is given by $N(s)=\alpha\inv\log (s)^{1-\alpha}$ for sufficiently large values of $s$, and in this case \eqref{eq:quasi-pos_inc} holds for $c=1$. Note that the function $M_N$ grows strictly more slowly than the function $\Mlog$ appearing in Theorem~\ref{thm:BD}, and hence even \eqref{eq:MNinv_ep} gives a sharper result than \eqref{eq:Mlog}.
\end{enumerate}
\end{ex}

Given a Banach space $X$ we say that a $C_0$-semigroup $\T$  on $X$ with generator $A$ is a \emph{quasi-multiplication semigroup} if 
$$\|T(t)R(\lambda,A)\|=\sup_{z\in\s(A)}\frac{e^{t\R z}}{|\lambda-z|},\quad t\ge0,$$
for every $\lambda\in\rho(A)$. This terminology is taken from \cite{BaChTo16}, although the definition given there is slightly more restrictive. It follows from the spectral theorem that any $C_0$-semigroup of normal operators is a quasi-multiplication semigroup, but the class also contains multiplication semigroups on non-Hilbertian function spaces.
Our next result describes the exact rate of decay for quasi-multiplication semigroups with arbitrary resolvent growth. The proof is an extension of the ideas used in Theorem~\ref{thm:inf_nec}; see also \cite[Proposition~5.1]{BaChTo16}. Recall that the spectral bound $s(A)$ of a semigroup generator $A$ is defined as $s(A)=\smash{\sup_{z\in\s(A)}\R z}$.

\begin{thm}\label{thm:normal_rate}
Let $X$ be a Banach space and let $A$ be the generator of a  quasi-multiplication semigroup $\T$ on $X$. Suppose that $s(A)=0$ but $\s(A)\cap i\RR=\emptyset$, and let $M\colon\RR_+\to(0,\infty)$ be defined by $M(s)=\smash{\sup_{|r|\le s}}\|R(ir,A)\|,$ $s\ge0$.
Then 
\begin{equation}\label{eq:Mmax_inv}
\|T(t)A\inv\| \sim \frac{1}{\Mmax\inv(t)},\quad t\to\infty,
\end{equation}
where $\Mmax\colon[1,\infty)\to\RR_+$ is defined by
\begin{equation*}\label{eq:Mmax}
\Mmax(s)=\max_{1\le\lambda\le s}M(\lambda\inv s)\log\lambda,\quad s\ge1.
\end{equation*}
\end{thm}

\begin{proof}
Since $\T$ is a quasi-multiplication semigroup we have
\begin{equation}\label{eq:spec_thm}
\|T(t)A\inv\|=\sup_{z\in\s(A)} \frac{e^{t\R z}}{|z|},\quad t\ge0,
\end{equation}
and also $M(s)=\smash{\sup_{|r|\le s}}\dist(ir,\s(A))\inv$, $s\ge0.$
In particular,  $M(s)\to\infty$ as $s\to\infty$ since $s(A)=0$. Now if $z\in\s(A)$ then $-\R z\ge M(|\I z|)\inv$, so 
$$\|T(t)A\inv\|\le\sup_{z\in\s(A)}\frac{1}{|z|}\exp\left(-\frac{t}{M(|\I z|)}\right),\quad t\ge0.$$
Since $M$ is unbounded we may assume, by choosing $t$ to be sufficiently large, that the supremum is unaffected by restricting consideration to points $z\in\s(A)$ satisfying $|\I z|\ge1$. Thus
\begin{equation}\label{eq:normal_rate}
\|T(t)A\inv\|\le\sup_{s\ge1}\frac1s
\exp\left(-\frac{t}{M(s)}\right)
\end{equation}
for all sufficiently large $t$. Given $t\ge 0$ let $R=\smash{\Mmax\inv}(t)$. Then for $s\ge R$ we have
$s\inv\exp(-tM(s)\inv)\le R\inv$, while for $1\le s \le R$  the definition of $\Mmax$ implies that $\Mmax(R)\ge M(s)\log(R/s)$
and hence again $s\inv\exp(-tM(s)\inv)\le R\inv$. Thus by \eqref{eq:normal_rate} we have $\|T(t)A\inv\|\le \smash{1/\Mmax\inv}(t)$ for all sufficiently large values of $t$.

Now let $\ep\in(0,1)$ and consider the function $K\colon\RR_+\to(0,\infty)$ defined by
$$K(t)=\frac{1-\ep}{\|T(t)A\inv\|},\quad t\ge0.$$
Note that, by \eqref{eq:spec_thm}, the function $K$ is continuous and strictly increasing. Arguing as in the proof of Theorem~\ref{thm:inf_nec} we see that for sufficiently large values of $s$ we may find $\alpha+i\beta\in\s(A)$ such that $-\alpha\le M(s)\inv$ and $|\alpha+i\beta|\le (1-\ep)\inv s$. It then follows as before from \eqref{eq:at_est} with $N\inv$ replaced by $K$, and with the choices $c=\delta=1$ and $C=1-\ep$, that there exists a constant $s_0>0$ such that $K\inv(\lambda s)\ge M(s)\log \lambda$ for all $\lambda\ge1$ and all $s\ge s_0$. Thus $K\inv(s)\ge M(\lambda\inv s)\log \lambda$, $1\le\lambda\le s/s_0$, whenever $s\ge s_0$. Using the fact that $M$ is unbounded, it is straightforward to see that for sufficiently large values of $s\ge s_0$ we have
$$\Mmax(s)=\max_{1\le \lambda\le\frac{s}{s_0}}M(\lambda\inv s)\log \lambda$$
 and therefore $K\inv(s)\ge \Mmax(s)$. Thus for $t$ sufficiently large we have $\Mmax\inv(t)\ge K(t)$, and hence
$$\|T(t)A\inv\|\ge\frac{1-\ep}{\Mmax\inv(t)}.$$
This completes the proof.
\end{proof} 

If we allow $s(A)<0$ in Theorem~\ref{thm:normal_rate} then it is  still true that
\begin{equation}\label{eq:conj}
\|T(t)A\inv\|=O\left(\Mmax\inv(t)\inv\right),\quad t\to\infty,
\end{equation}
as can be seen from a straightforward extension of the first part of the proof. However, in this case \eqref{eq:Mmax_inv} no longer holds in general. For instance, if we let $A$ be the generator of a quasi-multiplication semigroup such that $-\alpha\in\s(A)\subseteq(-\infty,-\alpha]$ for some $\alpha>0$, then $\|T(t)A\inv\|=\alpha\inv e^{-\alpha t}$ but $\smash{\Mmax\inv}(t)\inv=e^{-\alpha t}$, $t\ge0$, so \eqref{eq:Mmax_inv} is violated unless $\alpha=1$. We leave open whether \eqref{eq:conj} holds for more general bounded $C_0$-semigroups $\T$ on a Hilbert space with generator $A$ satisfying $\s(A)\cap i\RR=\emptyset$. Note that one does not in general have $\|T(t)A\inv\|\asymp \smash{\Mmax\inv}(t)\inv$, $t\to\infty$, as can be seen by letting $A$ be a $2\times2$ Jordan block.
We conclude this section by revisiting the special cases considered in Example~\ref{ex:MN}.

\begin{ex}\phantomsection\label{ex:normal} 
\begin{enumerate}[(a)]
\item For the function $M$ considered in part~\eqref{ex:log^a} of Example~\ref{ex:MN} a simple calculation shows that $\Mmax(s)=c_\alpha\inv\log (s)^{\alpha+1}$ for large values of $s$, where $c_\alpha$ is as before. In particular, \eqref{eq:ep_loss} gives the best possible estimate up to the arbitrarily small loss in the constant multiplying $t$. 
The sharper estimate \eqref{eq:refined} becomes 
\begin{equation}\label{eq:error}
 \|T(t)A\inv\| =O\left( \frac{t^{\frac{3}{2(\alpha+1)}}}{\Mmax\inv(t)}\right),\quad t\to\infty.
\end{equation}
We do not know whether the polynomial factor is really needed or whether perhaps the sharper estimate \eqref{eq:conj} holds in this case.
\item Let $M$ and $N$ be the functions considered in part~\eqref{ex:explog} of Example~\ref{ex:MN}, and recall that $c=1$ in this case. Direct estimates show that $\Mmax(s)\sim e\inv M_N(s)$ as $s\to\infty$, and in particular for any $\ep>0$ we have $\smash{\Mmax\inv}(t)\le \smash{M_N\inv}(e(1+\ep)t)$ for all sufficiently large $t$. In fact, for $\alpha\in(1/2,1)$ it is possible to show that $\smash{\Mmax\inv}(t)\sim \smash{M_N\inv}(et)$ as  $t\to\infty$. Hence in this case \eqref{eq:MNinv_ep} gives the best possible estimate up to the arbitrarily small loss in the constant multiplying $t$, and one may apply \eqref{eq:MNinv} or \eqref{eq:simple} to get a sharper estimate. However, for $\alpha\in[1/2,1)$ the function $\smash{\Mmax\inv}(t)$ grows strictly faster than $\smash{M_N\inv}(et)$ as $t\to\infty$, so \eqref{eq:MNinv_ep} would not give the best possible rate of decay even if we were allowed to set $\ep=0$. In this example it is possible to push our approach slightly further by allowing the choice of the auxiliary function $N$ and of the constant $c$ in \eqref{eq:quasi-pos_inc} to depend on $s$, but we do not pursue this idea here.
\end{enumerate}
\end{ex}

\section{Application to a wave equation with viscoelastic damping}\label{sec:wave}

In this section we apply the theoretical results of Section~\ref{sec:opt_dec} to obtain sharp estimates on the rate of energy decay for  solutions of a wave equation subject to damping at the boundary. Indeed, let us consider the problem
\begin{equation}\label{eq:wave}
\left\{
\begin{aligned}
u_{tt}(s,t)- \Delta u(s,t)&=0,\quad & s\in(0,1),\,t\in\RR,\\
\partial_n u(s,t)+k*u_t(s,t)&=0,&s\in\{0,1\},\, t\in\RR.
\end{aligned}
\right.
\end{equation}
Here $\partial _n$ denotes the outward normal derivative in the space variable at the boundary, the convolution is with respect to the time variable and $k\colon\RR_+\to\RR$ is a completely monotone integrable function, which is to say that there exists a positive Radon measure $\nu$ on $\RR_+$, satisfying $\nu(\{0\})=0$ and
$\smash{\int_{\RR_+}}\tau\inv\,\dd\nu(\tau)<\infty$, such that
\begin{equation}\label{eq:k}
k(t)=\int_{\RR_+}e^{-\tau t}\,\dd\nu(\tau),\quad t\in\RR_+.
\end{equation}
We extend $k$ to the whole real line by zero, and we assume throughout that $k\ne0$. This system can be viewed as a model of sound propagation under reflection subject to viscoelastic damping at the boundary, and in this case the boundary condition captures memory effects, $u_t$ and $-\nabla u$ are the pressure and velocity of the fluid and $\F k$, or alternatively the Laplace transform of $k$, is the \emph{acoustic impedance}; for further details see \cite{Sta17}, where the same model is considered also for higher-dimensional domains.  The results in this section are closely related to those obtained independently in \cite{BeRa17}, where rates of energy decay are investigated for a very similar model.

 We begin by recasting the problem in the form of an abstract Cauchy problem,
\begin{equation}\label{eq:ACP}
\left\{
\begin{aligned}
\dot{z}(t)&=Az(t),\quad  t\ge0,\\
z(0)&=x,
\end{aligned}
\right.
\end{equation}
where the initial data vector $x$ is an element of some Hilbert space $X$ and represents not only the pressure and velocity of the fluid at time $t=0$ but also the fluid pressure at the boundary  for all times $t<0$. It is shown in \cite{DeFaMiPr10} that for suitable choices of $A$ and of the Hilbert space $X$ this abstract Cauchy problem is well-posed and that the $C_0$-semigroup $\T$ generated by $A$ is contractive. Moreover, the square of the norm in the Hilbert space $X$ can be interpreted physically as the energy of the system. The following result is proved in \cite{Sta17}.

\begin{thm}\label{thm:wave_res}
Suppose that $\nu([0,\ep))=0$ for some $\ep>0$. Then $\s(A)\cap i\RR=\emptyset$ and there exist constants $C,c>0$ such that 
$$\frac{c}{\R\F k(s)}\le \sup_{|r|\le s} \frac{1}{\dist(ir,\s(A))} \le \sup_{|r|\le s}\|R(ir,A)\|\le \frac{C}{\R\F k(s)},\quad s\ge0.$$
\end{thm}

\begin{rem}
The assumption on $\nu$ ensures that $0\not\in\s(A)$. If this condition is not satisfied for any $\ep>0$ then $0\in\s(A)\subseteq\CC_-\cup\{0\}$ and $\|R(is,A)\|\asymp|s|\inv$ as $|s|\to0$; see \cite{Sta17}. Hence the model can also give rise to resolvents which have singularities at both zero and infinity. For simplicity we focus here only on the case where there is no singularity at zero.
\end{rem}

In view of Theorem \ref{thm:wave_res} it is natural to introduce the function $M\colon\RR_+\to(0,\infty)$ defined by
\begin{equation}\label{eq:M_def}
M(s)=\frac{1}{\R\F k(s)},\quad s\ge0.
\end{equation}
We have
$$\R\F k(s)=\int_{\RR_+}\frac{\tau}{\tau^2+s^2}\,\dd\nu(\tau),\quad s\ge0,$$
so the function $M$ is well-defined, continuous,  non-decreasing and  satisfies $M(s)\to\infty$ as $s\to\infty$. We now turn to the study of energy decay for classical solutions of \eqref{eq:ACP}. By combining Theorem~\ref{thm:wave_res} with Theorem~\ref{thm:BD} and the subsequent remarks we obtain the following result.

\begin{thm}\label{thm:wave_Mlog}
Suppose that $\nu([0,\ep))=0$ for some $\ep>0$ and define $M\colon\RR_+\to(0,\infty)$ as in \eqref{eq:M_def}. Then there exist constants $C,c>0$ such that
\begin{equation}\label{eq:wave_dec}
\frac{c}{M\inv(Ct)}\le \|T(t)A\inv\|\le\frac{C}{\Mlog\inv(ct)}
\end{equation}
for all sufficiently large values of $t$, where $\Mlog\colon\RR_+\to(0,\infty)$ is defined by $\Mlog(s)=M(s)(\log(1+s)+\log(1+M(s)))$, $s\ge0$.
\end{thm}

The following result characterises in terms of a simple condition on the acoustic impedance $\F k$ of our system those cases in which $\smash{\Mlog\inv}$ may be replaced by $M\inv$ in \eqref{eq:wave_dec}. It is a direct consequence of Theorems~\ref{thm:inf}, \ref{thm:inf_nec}, \ref{thm:wave_res} and Proposition~\ref{prp:M_inv}.

\begin{thm}\label{thm:wave_Minv}
Suppose that $\nu([0,\ep))=0$ for some $\ep>0$ and define $M\colon\RR_+\to(0,\infty)$ as in \eqref{eq:M_def}. If $M$ has positive increase then
\begin{equation}\label{eq:wave_Minv}
\|T(t)A\inv\|=O\left(M\inv(ct)\inv\right),\quad t\to\infty,
\end{equation}
for all $c>0$, and conversely if \eqref{eq:wave_Minv} holds for some $c>0$ then $M$ has positive increase.
\end{thm}

Examples given in \cite{Sta17} show that there exist suitable functions $k$ such that $\R\F k(s)\asymp s^{-\alpha}$, $s\to\infty$, for any $\alpha\in(0,1)$. In this case $M(s)\asymp s^\alpha$ as $s\to\infty$ and Theorem~\ref{thm:wave_Minv} certainly applies, but the same optimal rate of decay could already have been obtained using \cite[Theorem~2.4]{BoTo10}. We conclude this paper with a result showing that the function $k$ in our model can be chosen in such a way that $\R\F k$ has the same asymptotic behaviour as $1/M$ for any  given regularly varying function $M\colon\RR_+\to(0,\infty)$ of index strictly between 0 and 2. Such cases are only very partially covered by the results in \cite{BaChTo16}, but fall squarely into the scope of Theorem~\ref{thm:wave_Minv} above. 

\begin{prp}\label{prp:wave_reg}
Let $\alpha\in(0,2)$ and suppose that $M\colon\RR_+\to(0,\infty)$ is a regularly varying function of index $\alpha$. Then there exists a positive Radon measure $\nu$ on $\RR_+$ with $\nu([0,1))=0$ such that the function $k\colon\RR_+\to(0,\infty)$ defined by \eqref{eq:k} is integrable and satisfies $\R\F k(s)\sim 1/M(s)$ as $s\to\infty.$
\end{prp}

\begin{proof}
Let $\ell\colon[1,\infty)\to(0,\infty)$ be a slowly varying function such that $M(s)= s^\alpha/\ell(s)$, $s\ge1$, and define $g\colon\RR_+\to\RR_+$ by $g(s)=0$, $s\in[0,1)$, and 
$$g(s)=\frac{(2-\alpha)s^{-\alpha}\ell(s)}{\Gamma(\frac\alpha2)\Gamma(2-\frac\alpha2)},\quad s\ge1.$$
Moreover,  let $\nu$ be the Radon measure on $\RR_+$ with Lebesgue density $g$. Then $\smash{\int_{\RR_+}\tau\inv\,\dd\nu(\tau)}<\infty$, so the function $k$  defined by \eqref{eq:k} is integrable. A simple application of Fubini's theorem  shows that
$$\R\F k(s)=\int_1^\infty\frac{\tau g(\tau)}{\tau^2+s^2}\,\dd\tau=\smash{\int_1^\infty}\frac{h(\tau)}{\tau+s^2}\,\dd\tau,\quad s\ge0,$$
where $h(s)=\frac12g(s^{1/2})$, $s\ge0$. Let $H(s)=\smash{\int_0^s h(\tau)\,\dd\tau}$, $s\ge0$. Then by \cite[Theorem~1.5.8]{BiGoTe89} we see that
$$H(s)\sim\frac{s^{1-\alpha/2}\ell(s^{1/2})}{\Gamma(\frac\alpha2)\Gamma(2-\frac\alpha2)}, \quad s\to\infty,$$
and hence $\R\F k(s)\sim 1/M(s)$ as $s\to\infty$ by \cite[Theorem~1.7.4]{BiGoTe89}.
\end{proof}

\subsection*{Acknowledgements}

A significant part of the work on this paper was carried out while D.S.\ and R.S. visited J.R.\ at IM PAN, Warsaw, in March 2017. D.S.\ and R.S.\ both gratefully acknowledge the financial support they received from IM PAN during this visit. All three authors moreover wish to thank Charles Batty, Ralph Chill, Yuri Tomilov and the anonymous referees for numerous helpful comments and suggestions.

\end{document}